\documentclass[10pt]{amsart}
\usepackage[normalem]{ulem}
\usepackage[usenames]{color}
\usepackage{amsrefs}
\usepackage{amssymb}
\usepackage[usenames,dvipsnames]{pstricks}
\usepackage{pstricks-add}

 \newcommand{\R}{\mathbb R}
 \newcommand{\Z}{\mathbb Z}
\newcommand{\ep}{\varepsilon}

\DeclareMathOperator{\Int}{Int}

\newcommand{\ds}{\displaystyle}

\theoremstyle{plain} \newtheorem{thm}{Theorem}
\newtheorem{cor}[thm]{Corollary} \newtheorem{prop}[thm]{Proposition}
\newtheorem{lemma}[thm]{Lemma}

\theoremstyle{definition} \newtheorem{defn}[thm]{Definition}

\newtheorem{ex}[thm]{Example} 

\theoremstyle{remark} \newtheorem{remark}[thm]{Remark}

\DeclareMathOperator{\gr}{GR}
\DeclareMathOperator{\crec}{CR}
\newcommand{\scrwo}{\operatorname{SCR}}
\newcommand{\screc}{\scrwo_d}
\DeclareMathOperator{\mane}{M}
\newcommand{\prodset}[1]{{\mathcal{#1}}}
\DeclareMathOperator{\id}{Id}
\newcommand{\diag}{\Delta_X}
\DeclareMathOperator{\Fix}{Fix}
\newcommand{\nw}{\prodset{NW}}
\newcommand{\nseq}{\Sigma}

\newcommand{\tube}{\prodset{V}}
\newcommand{\opentube}{\prodset{V}^\circ}
\newcommand{\dgr}{d^\ast}
\newcommand{\mrel }{\prodset{M}}
\newcommand{\wrel }{\prodset{W}}
\newcommand{\arel }{\prodset{A}}

\numberwithin{thm}{section}

\begin{document}

\title{The Generalized Recurrent Set and Strong Chain Recurrence}

  \author{Jim Wiseman}   \address{Agnes Scott College \\ Decatur, GA 30030} \email{jwiseman@agnesscott.edu}

\thanks{This work was supported by a grant from the Simons Foundation (282398, JW)}


\begin{abstract}
Fathi and Pageault have recently shown a connection between Auslander's generalized recurrent set $\gr(f)$ and Easton's strong chain recurrent set.  We study $\gr(f)$ by examining that connection in more detail, as well as connections with other notions of recurrence.   We give equivalent definitions that do not refer to a metric.  In particular, we show that $\gr(f^k)=\gr(f)$ for any $k>0$, and give a characterization of maps for which the generalized recurrent set is different from the ordinary chain recurrent set.
\end{abstract}

\maketitle

\section{Introduction}

Auslander's generalized recurrent set $\gr(f)$ (defined originally for flows (see \cite{Auslander}), and extended to maps  (see \cites{A,AA})) is an important object of study in dynamical systems.  (See, for example, \cites{Nit,Nit2,Peix1,Peix2,ST,ST2,AusNC,KK,Garay}.)  Fathi and Pageault have recently shown (\cite{FP}) that $\gr(f)$ can be defined in terms of Easton's strong chain recurrent set (\cite{E}) (although they did not use the strong chain recurrent terminology).  (See  \cites{ABC,Y} for more on the literature on the strong chain recurrent set.)  In this paper we study the generalized recurrent set by
examining that connection in more detail, as well as connections with other notions of recurrence.    In particular, we show that $\gr(f^k)=\gr(f)$ for any $k>0$, and give a characterization of maps for which the generalized recurrent set is different from the ordinary chain recurrent set.

The strong chain recurrent set depends on the choice of metric, and thus Fathi and Pageault's description of $\gr(f)$ involves metrics.  Since the generalized recurrent set itself is a topological invariant, it is useful to be able to describe it in terms of strong chain recurrence without referring to a metric  (especially in the noncompact case, as in \cites{AA}).  We give definitions with topological versions of strong $\ep$-chains that do not involve a metric.

The paper is organized as follows.  We give definitions and examples in Section~\ref{sect:defns}, and discuss Fathi and Pageault's Ma\~n\'e set in Section~\ref{sect:mane}.  In Section~\ref{sect:GR} we turn to the generalized recurrent set, giving a topological definition and showing, in particular,  that there exists a metric for which the strong chain recurrent set equals $\gr(f)$.  In Section~\ref{sect:powers} we show that $\gr(f^k)=\gr(f)$ for any $k>0$.  Finally, in Section~\ref{sect:relation} we consider the relationship between the generalized recurrent set and the ordinary chain recurrent set.

Thanks to Todd Fisher and David Richeson for useful conversations on these topics, and to the anonymous referee for very prompt and helpful  comments and perspective. Among other things, the referee provided a greatly improved proof of Theorem~\ref{thm:mwequal}.

\section{Definitions and examples} \label{sect:defns}

Throughout this paper, let  $(X,d)$ be a compact metric space and $f:X\to X$  a continuous map.  Recurrence on noncompact spaces is more complicated and will be the subject of future work.

\begin{defn} 
An {\em $(\ep,f,d)$-chain} (or {\em $(\ep,d)$-chain}, if it is clear what the map is, or \emph{$\ep$-chain}, if the metric is also clear) of length $n$  from $x$ to $y$ is a sequence $(x=x_0, x_1, \dots, x_n=y)$ such that $d(f(x_{i-1}),x_i)\le\ep$ for $i=1,\dots,n$.  A point $x$ is {\em chain recurrent} if for every $\ep>0$, there is an $\ep$-chain from $x$ to itself.  We denote the set of chain recurrent points by $\crec(f)$.  
Two points $x$ and $y$ in $\crec(f)$ are \emph{chain equivalent} if there are $\ep$-chains from $x$ to $y$ and from $y$ to $x$ for any $\ep>0$.
The map $f$ is  \emph{chain transitive} on a subset $N$ of $X$ if for every $x,y\in N$ and every $\ep>0$, there is an $\ep$-chain from $x$ to $y$; the chain equivalence classes are called the \emph{chain transitive components}.
\end{defn}

\begin{remark}
Chain recurrence depends only on the topology, not on the choice of metric (see, for example, \cite{Franks}).
\end{remark}

The following definitions are due to Easton~\cite{E}.
\begin{defn} 
A {\em strong $(\ep,f,d)$-chain} (or {\em strong $(\ep,d)$-chain} or \emph{strong $\ep$-chain})  from $x$ to $y$ is a sequence $(x=x_0, x_1, \dots, x_n=y)$ such that $\sum_{i=1}^n d(f(x_{i-1}),x_i)\le\ep$.  A point $x$ is {\em $d$-strong chain recurrent} (or \emph{strong chain recurrent}) if for every $\ep>0$, there is a strong $(\ep,d)$-chain from $x$ to itself.  We denote the set of strong chain recurrent points by $\screc(f)$.  
Two points $x$ and $y$ in $\screc(f)$ are \emph{$d$-strong chain equivalent} (or \emph{strong chain equivalent}) if there are strong $(\ep,d)$-chains from $x$ to $y$ and from $y$ to $x$ for any $\ep>0$.
A subset $N$ of $X$ is \emph{$d$-strong chain transitive} (or \emph{strong chain transitive}) if  every $x$ and $y$ in $N$ are $d$-strong chain equivalent;  the strong chain equivalence classes are called the \emph{strong chain transitive components}.
\end{defn}

\begin{ex} \label{ex:halfcirc}
Let $X_1$ be the circle with the usual topology, and let $f_1:X_1\to X_1$ be a homeomorphism that fixes every point on the left semicircle $C_1$ and moves points on the right semicircle clockwise (see Figure~\ref{fig:halfcirc}).  Then for any choice of metric $d$, we have $\screc(f_1)=C_1$, and each point in $C_1$ is a strong chain transitive component.
\end{ex}

\begin{figure}
	\begin{center}
\psscalebox{1.0 1.0} 
{
\begin{pspicture}(0,-1.66)(3.28,1.66)
\psarc[arrowsize=3pt 2]{<-}(1.66,0.0){1.6}{-2}{90.0}
\psarc(1.66,0.0){1.6}{270.0}{370.0}
\psarc[linecolor=black, linewidth=0.06, dimen=outer](1.66,0.0){1.6}{90.0}{270.0}
\end{pspicture}
}
\end{center}
\caption{$f_1:X_1\to X_1$}
\label{fig:halfcirc}
	\end{figure}

\begin{remark}
In general, strong chain recurrence does depend on the choice of metric.  See Example 3.1 in \cite{Y}, or the following example from \cite{FP}.
\end{remark}

\begin{ex}[\cite{FP}] \label{ex:cantor}
Consider the circle with the usual topology, and a map that fixes a Cantor set   and moves all other points clockwise (see Figure~\ref{fig:cantor}).  Choose a metric $d_2$ for which the Cantor set has Lebesgue measure 0; call the resulting metric space $X_2$,  the map $f_2$, and the Cantor set $K_2$.  Then $\scrwo_{d_2}(f_2) = X_2$.  Or we can choose a metric $d_3$ for which the Cantor set has positive Lebesgue measure, and call the resulting metric space $X_3$, with map $f_3$ and  Cantor set $K_3$.  Then $\scrwo_{d_3}(f_3) = K_3$.
\end{ex}

\begin{figure}
\psscalebox{1.0 1.0} 
{
\begin{pspicture}(0,-1.66)(3.28,1.66)
\psarc(1.66,0.0){1.6}{0.0}{360.0}
\psarc[linecolor=black, linewidth=0.08, dimen=outer](1.66,0.0){1.6}{30.0}{36.67}
\psarc[linecolor=black, linewidth=0.08, dimen=outer](1.66,0.0){1.6}{43.33}{50.0}
\psarc[linecolor=black, linewidth=0.08, dimen=outer](1.66,0.0){1.6}{70.0}{76.67}
\psarc[linecolor=black, linewidth=0.08, dimen=outer](1.66,0.0){1.6}{83.33}{90.0}
\psarc[linecolor=black, linewidth=0.08, dimen=outer](1.66,0.0){1.6}{-36.67}{-30.0}
\psarc[linecolor=black, linewidth=0.08, dimen=outer](1.66,0.0){1.6}{-50.0}{-43.33}
\psarc[linecolor=black, linewidth=0.08, dimen=outer](1.66,0.0){1.6}{-76.67}{-70.0}
\psarc[linecolor=black, linewidth=0.08, dimen=outer](1.66,0.0){1.6}{-90.0}{-83.33}
\psarc[arrowsize=3pt 2]{<-}(1.66,0.0){1.6}{-2}{-2.1}
\psarc[arrowsize=3pt 2]{<-}(1.66,0.0){1.6}{55}{54.9}
\psarc[arrowsize=3pt 2]{<-}(1.66,0.0){1.6}{-65}{-65.1}
\psarc[arrowsize=3pt 2]{<-}(1.66,0.0){1.6}{-182}{-182.1}

\end{pspicture}
}
\caption{$f_2:X_2\to X_2$ and $f_3:X_3\to X_3$}
\label{fig:cantor}
\end{figure}

\begin{remark}
Fathi and Pageault~\cite{FP} define a function $L_d:X\times X\to [0,\infty]$, which they call the {\em $d$-Mather barrier}, by $L_d(x,y)= \inf \sum_{i=1}^n d(f(x_{i-1}),x_i)$, where the infimum is over all sequences $(x=x_0, x_1, \dots, x_n=y)$ from $x$ to $y$. (Zheng used a similar function in \cite{Zheng1}.) They then define the {\em $d$-Aubry set} to be $\{x\in X : L_d(x,x)=0\}$.  Thus their $d$-Aubry set is identical to Easton's strong chain recurrent set.  Similarly, they define an equivalence relation on the $d$-Aubry set by setting $x$ and $y$ equivalent if $L_d(x,y) = L_d(y,x)=0$, and call the equivalence classes $d$-Mather classes.  Thus the $d$-Mather classes are exactly the $d$-strong chain transitive components.
\end{remark}

To eliminate the dependence on the metric in $\screc$, we can take either the intersection or the union over all metrics, giving us two different sets.

\begin{defn}[\cite{FP}]
The \emph{Ma\~n\'e set $\mane(f)$} is $\bigcup_{d'} \scrwo_{d'}(f)$ and the \emph{generalized recurrent set $\gr(f)$} is $\bigcap_{d'} \scrwo_{d'}(f)$, where the union and the intersection are both over all metrics $d'$ compatible with the topology of $X$.  (Fathi and Pageault show (\cite{FP}) that this definition of the generalized recurrent set is equivalent to the usual definitions; see Section~\ref{sect:GR}.)
\end{defn}

Thus we have $\gr(f) \subset \screc(f) \subset \mane(f) \subset \crec(f)$; all of the inclusions can be strict, as the following example shows.

\begin{ex}
Let $X$ be the disjoint union of the spaces $X_1$, $X_2$, and $X_3$ from Examples~\ref{ex:halfcirc} and \ref{ex:cantor}, with the induced metric $d$.  Define the map $f:X\to X$ by $f(x)=f_i(x)$ for $x\in X_i$.  Then we have $\gr(f) =C_1 \cup K_2 \cup K_3$, $\screc(f) = C_1 \cup X_2 \cup K_3$, $\mane(f) = C_1 \cup X_2 \cup X_3$, and $\crec(f) = X_1 \cup X_2 \cup X_3$.
\end{ex}

\section{The Ma\~n\'e set $\mane(f)$} \label{sect:mane}

 We  give an equivalent definition of the Ma\~n\'e set $\mane(f)$ based on strong $\ep$-chains, but using a topological definition of chains that does not depend on the metric (Corollary~\ref{cor:newmane}).
  We begin with some notation.  Let $X\times X$ be the product space, and let $\diag$ be the diagonal, $\diag=\{(x,x):x \in X\}$.  To avoid confusion, we will use calligraphic letters like $\prodset{N}$ for other subsets of $X\times X$, and reserve italic letters like $N$ for subsets of $X$.  
  
  Let $B_d(x;\ep)$ (or $B(x;\ep)$ if the metric is clear) be the closed $\ep$-ball around $x$, $B_d(x;\ep) =\{y\in X : d(x,y)\le\ep\}$.  Let $\tube_d(\ep)$ (or $\tube(\ep)$) be the closed $\ep$-neighborhood of the diagonal $\diag$ in $X\times X$, $\tube_d(\ep) = \{(x_1,x_2) : d(x_1,x_2)\le\ep\}$, and $\opentube_d(\ep)$ (or $\opentube(\ep)$)  the open $\ep$-neighborhood, $\opentube_d(\ep) = \{(x_1,x_2) : d(x_1,x_2)<\ep\}$.
  
  For $\prodset{N}\subset X\times X$, 
  we denote by $\prodset{N}^n$ the $n$-fold composition of $\prodset{N}$ with itself, $\prodset{N}\circ\prodset{N}\cdots\circ\prodset{N}$, that is, 
  \begin{align*}
  \prodset{N}^n = & \{(x,y) : \text{there exists $z_0=z,z_1,\ldots,z_n=y\in X$} \\ & \text{ such that $(z_{i-1},z_i)\in\prodset{N}$ for $i=1,\ldots,n$}\}.
  \end{align*}

\begin{defn}\label{defn:nchain}
Let $\prodset{N}$ be a neighborhood of $\diag$.  An \emph{$(\prodset{N},f)$-chain} (or simply \emph{$\prodset{N}$-chain} if the map is clear) from $x$ to $y$ is a sequence of points $(x=x_0, x_1, \dots, x_n=y)$ in $X$ such that $(f(x_{i-1}),x_i) \in \prodset{N}$ for $i=1,\ldots,n$.
\end{defn}

Thus $(x,y)\in \prodset{N}^n$ exactly when there is an $(\prodset{N},\id)$-chain of length $n$ from $x$ to $y$, where $\id$ is the identity map.

\begin{defn}\label{defn:mrelations}
We now define three relations on $X$.  We write $y>_{d'}z$ if for any $\ep>0$, there is a strong $(\ep,f,d')$-chain from $y$ to $z$.  We write $y >_\mrel z$ if $y>_{d'}z$ for some compatible metric $d'$; set $\mrel = \{(y,z)\in X \times X: y >_\mrel z\}$.
We write $y >_\wrel z$ if for any closed neighborhood $\prodset{D}$ of the diagonal in $X\times X$, there exist a closed symmetric neighborhood $\prodset{N}$ of the diagonal and an integer $n>0$ such that $\prodset{N}^{3^n} \subset \prodset{D}$ and there is an $(\prodset{N},f)$-chain of length $n$ from $y$ to $z$; set $\wrel = \{(y,z)\in X \times X: y >_\wrel z\}$.
\end{defn}

\begin{thm}\label{thm:mwequal}
The relations $\mrel$ and $\wrel$ are equal.
\end{thm}

\begin{proof}
We will show that $\mrel \subset \wrel \subset \overline\wrel \subset \mrel$ (where $\overline\wrel$ is the closure of $\wrel$ in $X\times X$), and so they are all equal.

We first show that $\mrel \subset \wrel$.  Let $(y,z)$ be a point in $\mrel$; then there is a metric $d'$ such that for any $\ep>0$, there is a strong $(\ep,f,d')$-chain from $y$ to $z$.  Given $\prodset D$, choose $\ep$ such that $\tube_{d'}(\ep) \subset \prodset D$.  (Such an $\ep$ exists since $X\times X $ is compact.)  Let $(x_0=y,x_1,\ldots,x_n=z)$ be a strong $(\ep/2,d')$-chain from $y$ to $z$.  For $1\le i \le n$, define $\ep_i = d'(f(x_{i-1}),x_i)$, and let $B_i = B_{d'}(x_i;\ep_i/2)$ (note that $B_i$ is the single point $\{x_i\}$ if $\ep_i =0$).  Finally, define $\prodset N$ by $\prodset{N}=\tube_{d'}(\frac{\ep}{2\cdot3^n}) \bigcup (\bigcup_{i=1}^n B_i\times B_i)$.

Since $(f(x_{i-1}),x_i)$ is in $B_i\times B_i$, $(x_0=y,x_1,\ldots,x_n=z)$ is an $(\prodset{N},f)$-chain.  To see that $\prodset{N}^{3^n} \subset \prodset{D}$, let $z_0, z_1,\ldots,z_{3^n}$ be a sequence with $(z_{j-1},z_j)\in\prodset{N}$ for $1\le j\le 3^n$; we want to show that $d'(z_0,z_{3^n}) \le \ep$.  Observe that if $z_j$ and $z_k$ are both in $B_i$ for some $i$ and some $j<k$, then $z_0, z_1,\ldots,z_{j-1},z_j,z_k,z_{k+1},\ldots,z_{3^n}$ is also an $(\prodset{N},\id)$-chain from $z_0$ to $z_{3^n}$, possibly of shorter length.  Thus we may assume that for each $B_i$, the chain contains  at most one pair of points in $B_i$ and that any two such points are adjacent in the chain; two adjacent points that are not in the same $B_i$ must be within $\frac{\ep}{2\cdot3^n}$ of each other.  Therefore $d'(z_0,z_{3^n}) \le 3^n\cdot \frac{\ep}{2\cdot 3^n} + \sum \ep_i \le \frac{\ep}{2}+\frac{\ep}{2}$.

To show that $\overline\wrel \subset \mrel$, we need the following metrization lemma.

\begin{lemma}[\cite{Kelley}*{Lemma~6.12}]
Let $\{\prodset U_n\}_{n=0}^\infty$ be a sequence of symmetric subsets of $X\times X$ with $\prodset U_0 = X\times X$ and $\bigcap_{n=0}^\infty \prodset U_n = \diag$.  If for every $n\ge1$, $\prodset U_n^3 \subset \prodset U_{n-1}$, then there exists a metric $d'$ on $X$ such that $\prodset U_n \subset \opentube_{d'}(2^{-n})\subset \prodset U_{n-1}$.
\end{lemma}

(The lemma actually says that there exists a pseudo-metric, but since $X$ is metrizable, any pseudo-metric is a metric.)

Let $(y,z)$ be a point in $\overline\wrel$; we will construct a metric $d'$, depending on $(y,z)$, such that $y >_{d'}z$ (and so $y >_\mrel z$).  We construct the sequence for the metrization lemma by induction.  Let $\prodset A_0 = X\times X$.  Then assume that a closed, symmetric neighborhood of the diagonal $\prodset A_k$ has been constructed.  Let $\prodset A_k'$ be a closed, symmetric neighborhood of the diagonal such that $(\prodset A_k')^3 \subset \prodset A_k$ and  $(f\times f)(\prodset A_k') \subset \prodset A_k$ (this is possible by compactness and uniform continuity).
We can choose $\prodset A_k'$ inside $\tube_d(1/n)$ to guarantee that the $\prodset A_k$'s will shrink to $\diag$.
  Since $(y,z) \in \overline\wrel$, there exists a point $(y_k,z_k) \in \wrel$ with $(y,y_k) \in \prodset A_k'$ and $(z,z_k) \in \prodset A_k'$.  Then there exist a closed symmetric neighborhood $\prodset A_{k+1}$ of the diagonal and an integer $n_k$ such that there is an $\prodset A_{k+1}$-chain of length $n_k$ from $y_k$ to $z_k$ and $(\prodset A_{k+1})^{3^{n_k}} \subset \prodset A_k'$.  Then we can apply the metrization lemma (after renumbering) to the sequence
$$
\prodset A_0, (\prodset A_1)^{3^{n_0}}, (\prodset A_1)^{3^{n_0-1}}, \dots, (\prodset A_1)^{3^{2}}, (\prodset A_1)^{3}, \prodset A_1,(\prodset A_2)^{3^{n_1}}, (\prodset A_2)^{3^{n_1-1}}, \dots
$$
to obtain the compatible metric $d'$.  For any $\ep>0$, choose $k$ so that $\prodset A_k \subset \opentube_{d'}(\ep/3)$; then $\prodset A_{k+1} \subset \opentube_{d'}(2^{-n_k}\ep/3)$.  If we take our $\prodset A_{k+1}$-chain of length $n_k$ from $y_k$ to $z_k$, $(y_k,x_1,\dots,x_{n_k-1},z_k)$, and change the beginning and ending points to get a chain $(x_0=y,x_1,\dots,x_{n_k-1},x_{n_k}=z)$ from $y$ to $z$, we have that $\sum_{i=1}^{n_k} d'(f(x_{i-1}),x_i) \le n_k\cdot(2^{-n_k}\ep/3) + d'(f(y),f(y_k)) + d'(z, z_k) \le \ep/3 + \ep/3 + \ep/3 = \ep$.

\end{proof}

\begin{cor}\label{cor:newmane}
A point $x\in X$ is in $\mane(f)$ if and only if for any closed neighborhood $\prodset{D}$ of the diagonal in $X\times X$, there exist a closed symmetric neighborhood $\prodset{N}$ of the diagonal and an integer $n>0$ such that $\prodset{N}^{3^n} \subset \prodset{D}$ and there is an $(\prodset{N},f)$-chain of length $n$ from $x$ to itself.
\end{cor}

\begin{proof}
Clearly $x\in\mane(f)$ if and only if $x >_\mrel x$.
\end{proof}

In particular, $\mane(f)$ is closed, since we saw that $\mrel$ is closed.

\begin{prop} \label{prop:mfnotmfk}
In general, $\mane(f|_{\mane(f)}) \ne \mane(f)$.
\end{prop}

\begin{proof}
See \cite{Y} (Example~3.1 and the examples constructed in Theorem~4.2), or the following example.
\end{proof}

\begin{ex}\label{ex:spiralhalf}
Let $X_4$ be the disk with the usual topology, and let $f_4:X_4\to X_4$ be a map that fixes the center point $(0,0)$ and the left outer semicircle $C_4$, moves points on the right outer semicircle clockwise, and moves interior points other than the center in a clockwise spiral out toward the outer circle $S_4$ (see Figure~\ref{fig:spiralhalf}).  Then $\mane(f_4) = \{(0,0)\} \cup S_4$, but $\mane(f_4|_{\mane(f_4)}) = \{(0,0)\} \cup C_4$.
\end{ex}

\begin{figure}
\psscalebox{1.0 1.0} 
{
\begin{pspicture}(0,-1.66)(3.28,1.66)
\psarc[arrowsize=3pt 2]{<-}(1.66,0.0){1.6}{-2}{90.0}
\psarc(1.66,0.0){1.6}{270.0}{370.0}
\psarc[linecolor=black, linewidth=0.06, dimen=outer](1.66,0.0){1.6}{90.0}{270.0}
\rput(1.66,0){\psplot[plotstyle=curve, plotpoints=100, xunit=1.0, yunit=1.0, polarplot=true,arrowsize=3pt 2,arrows=->]{1.0}{-724.0}{-1.6 x 20 div sin mul}}
\rput(1.66,0){\psplot[plotstyle=curve, plotpoints=100, xunit=1.0, yunit=1.0, polarplot=true]{-720}{-1890.0}{-1.6 x 20 div sin mul}}
\psdots[linecolor=black, dotsize=0.10](1.66,0)
\end{pspicture}
}
\caption{$f_4:X_4\to X_4$}
\label{fig:spiralhalf}
\end{figure}

Fathi and Pageault show (\cite{FP}*{Thm.~3.5}) that for homeomorphisms, $\mane(f)= \Fix(f) \cup \crec(f|_{X\backslash\Int(\Fix(f))})$.
Thus $\mane(f)$ depends strongly on the set of fixed points, but not on the other periodic points.  This can lead to counterintuitive results, as the following example shows.

\begin{ex} \label{ex:twocircs}
Let $f_1:X_1\to X_1$ be the homeomorphism from Example~\ref{ex:halfcirc}.  Define the space $X = X_1\times \Z_2$ and the homeomorphism $f:X\to X$ by $f(x,0) = (f_1(x), 1)$ and $f(x,1) = (f_1(x), 0)$.  Then $f$ has no fixed points, and so  we have $\mane(f) = \crec(f)= X$, which is somewhat counterintuitive since $f$ is just two copies of $f_1$ and $\mane(f_1) = C_1$, the left semicircle.  By the definition of $\mane(f)$, for every point in $X$, there must be a metric $d$ such that $x \in \screc(f)$.  One can show that if we give $X_1\times\{0\}$ the usual Euclidean metric, and  $X_1\times\{1\}$ the usual metric on the left semicircle and the metric induced by the Minkowski ?  function (\cite{Mink}) on the right semicircle, we get $\screc(f)=X$.
\end{ex}

Thus $M(f)$ occupies a middle ground between $\crec(f)$ and $\gr(f)$, and is perhaps of less dynamical interest than either, so we now turn to $\gr(f)$.

\section{The generalized recurrent set $\gr(f)$} \label{sect:GR}

Part of the usefulness of the generalized recurrent set $\gr(f)$ stems from the fact that it can be defined in terms of several different dynamical concepts.  As we have seen, Fathi and Pageault give a definition in terms of the strong chain recurrent set, and we will give one using a topological version of strong $\ep$-chains.  We begin by reviewing existing results.

Following the notation in \cite{FP}, let $\theta:X\to\R$ be a Lyapunov function for $f$ (that is, $\theta(f(x))\le \theta(x)$ for all $x$), and let $N(\theta)$ be the set of neutral points, that is, $N(\theta)=\{x\in X: \theta(f(x))=\theta(x)\}$.  Denote by $L(f)$ the set of continuous Lyapunov functions for $f$, and by $L_{d'}(f)$ the set of Lipschitz (with respect to the metric $d'$) Lyapunov functions for $f$.

\begin{prop}[\cites{AA,FP,A}] \label{prop:gre}
The following definitions for the generalized recurrent set $\gr(f)$ are equivalent.
	\begin{enumerate}
	\item (\cite{FP}) \label{item:fp} $\bigcap_{d'} \scrwo_{d'}(f)$, where the intersection is over all metrics $d'$ compatible with the topology of $X$.
	\item (\cite{FP}) $\bigcap_{d'}\bigcap_{\theta\in L_{d'}(f)} N(\theta)$, where the outer intersection is over all metrics $d'$ compatible with the topology of $X$.
	\item (\cites{A,AA}) $\bigcap_{\theta\in L_{(f)}} N(\theta)$.
	\item (\cites{A,AA}) \label{item:smallest}The set of points $x$ such that $(x, x)$ is an element of the smallest closed, transitive relation containing the graph of $f$.
	\item (\cites{A,AA}) \label{item:tf}The set of points $x$ such that $(x, x)$ is an element of $\mathcal Gf$, where $\mathcal Gf$ is as defined below.
	\end{enumerate}
\end{prop}

\begin{defn}[\cites{A,AA}] \label{defn:gf}
$\mathcal Gf$ is defined using transfinite recursion.  For any subset $\prodset{R}$ of $X\times X$, define its orbit $\prodset{O(R)}$ by $\prodset{O(R)}= \bigcup_{i\ge1}\prodset{R}^i$, and define $\nw(\prodset{R})$ to be $\overline{\prodset{O(R)}}$ (the closure, in $X\times X$, of $\prodset{O(R)}$).  
Let $\nw_0(f)$ be the graph of $f$, that is, $\nw_0(f)=\{(x,f(x)): x \in X\}$,
 and define inductively $\nw_{\alpha+1}(f) = \nw(\nw_\alpha(f))$ for $\alpha$ an ordinal number and $\nw_\beta(f)=\overline{\bigcup_{\alpha<\beta}\nw_\alpha(f)}$ for $\beta$ a limit ordinal.  This will stabilize at some countable ordinal $\gamma$, and we define $\mathcal Gf$ to be $\nw_\gamma(f)$.
 Note that $\mathcal Gf$ is the smallest closed, transitive relation containing the graph of $f$ referred to in Proposition~\ref{prop:gre}(\ref{item:smallest}).
\end{defn}

Again, we give a definition based on strong $\ep$-chains, but using a topological definition of chains that does not depend on the metric.

\begin{defn}
Let $\nseq = \{\prodset{N}_i\}_{i=1}^\infty$ be a sequence of neighborhoods of the diagonal $\diag$.  
A $(\nseq,f)$-chain (or simply $\nseq$-chain) is a finite sequence of points $(x=x_0, x_1, \dots, x_n=y)$ in $X$ such that $(f(x_{i-1}),x_i) \in \prodset{N}_{\sigma(i)}$ ($i=1,\ldots,n$)  for some   injection $\sigma: \{1,\ldots,n\}\to\mathbb{N}$.  (The injection $\sigma$ is the same for all $i$.)  Note that since $\sigma$ is one-to-one, each neighborhood $\prodset{N}_i$ can be used at most once in any $\nseq$-chain.
\end{defn}

\begin{thm} \label{thm:newgrdef}
A point $x\in X$ is in $\gr(f)$ if and only if for any sequence $\nseq$ of neighborhoods of the diagonal $\diag$, there exists a $(\nseq,f)$-chain from $x$ to $x$.
\end{thm}

\begin{proof}

We prove a slightly stronger result, in terms of relations.  As in Definition~\ref{defn:mrelations}, we write $y>_{d'}z$ if for any $\ep>0$, there is a strong $(\ep,f,d')$-chain from $y$ to $z$.  We write $y >_\arel z$ if $y>_{d'}z$ for all compatible metrics $d'$, and set $\arel = \{(y,z)\in X \times X: y >_\arel z\}$.
We write $y>_\prodset{C} z$ if there is a $\Sigma$-chain from $y$ to $z$ for any sequence $\nseq$ of neighborhoods of $\diag$, and set $\prodset{C} = \{(y,z)\in X\times X: y >_\prodset{C} z\}$.  
We will show that $\mathcal Gf = \prodset C = \arel$, by proving that $\mathcal Gf \subset \prodset C \subset \arel \subset \mathcal Gf$.  We begin with the following lemma.

\begin{lemma} \label{lem:closed}
The set $\prodset{C}$ is closed in $X\times X$.
\end{lemma}

	\begin{proof}[Proof of Lemma~\ref{lem:closed}]

Let $\{(y_j,z_j)\}_{j=1}^{\infty}$ be a sequence of points in $\prodset{C}$ with $\ds\lim_{j\to\infty}(y_j,z_j)=(y,z)$; we must show that $(y,z)\in\prodset{C}$.

First, observe that if $\Sigma'$ is a subsequence of $\Sigma$, then any $\Sigma'$-chain is also a $\Sigma$-chain.  Similarly, if $\prodset{N}_i' \subset \prodset{N}_i$ for all $i$, then any $\{\prodset{N}_i' \}_{i=1}^\infty$-chain is also a $\Sigma$-chain.

Let $\nseq = \{\prodset{N}_i\}_{i=1}^\infty$ be any sequence of neighborhoods of $\diag$. For $i=1$ and $2$, choose 
 $\widetilde{\prodset{N}_i}$ to be a neighborhood of the diagonal small enough that $\widetilde{\prodset{N}_i}^2 \subset \prodset{N}_i$.  Choose a $K$ large enough that $(f(y), f(y_K))\in \widetilde{\prodset{N}_1}$ and $(z_K, z) \in \widetilde{\prodset{N}_2}$.  Define a new sequence $\nseq' = \{\prodset{N}_i \cap \widetilde{\prodset{N}_1} \cap \widetilde{\prodset{N}_2}\}_{i=3}^\infty$.  Since $(y_K,z_K)\in \prodset{C}$, there is a $\nseq'$-chain $(x_0=y_K, x_1,\ldots, x_n=z_K)$ from $y_K$ to $z_K$.  Thus $(f(y), f(y_K))\in \widetilde{\prodset{N}_1}$ and $(f(y_K),x_1)\in \widetilde{\prodset{N}_1}$, so $(f(y),x_1) \in \widetilde{\prodset{N}_1}^2 \subset \prodset{N}_1$.  Similarly, $(f(x_{n-1}), z_K) \in \widetilde{\prodset{N}_2}$ and   $(z_K, z) \in \widetilde{\prodset{N}_2}$, so $(f(x_{n-1}),z) \in \widetilde{\prodset{N}_2}^2 \subset \prodset{N}_2$.  Therefore $(y, x_1,\ldots,x_{n-1},z)$ is a $\nseq$-chain from $y$ to $z$.  Since $\nseq$ was arbitrary, we have $(y,z)\in\prodset{C}$.

	\end{proof}

The relation $\prodset C$ clearly contains the graph of $f$ and is transitive, so $\mathcal Gf \subset \prodset C$ by Proposition~\ref{prop:gre}(\ref{item:smallest}).

Next we show that $\prodset C \subset \arel$.  Take $y>_{\prodset C} z$, and let $d'$ be any compatible metric and $\ep$ any positive number.  Define the sequence $\nseq=\{\prodset{N}_i\}_{i=1}^\infty$ by 
 $\prodset{N}_i=\tube_{d'}(\ep/2^i)$.
 Then any $\nseq$-chain is a strong $(\ep,d')$-chain.  Since $\ep$ was arbitrary, we have $y>_{d'}z$; since $d'$ was arbitrary, we have $y >_\arel z$, as desired.
 
 Finally, we show that $\arel \subset \mathcal Gf$.  Let $(y,z)$ be a point in $\arel$.  We first consider $(y,z) \in \arel$ with $y\ne z$, and let $\theta$ be a continuous Lyapunov function for $f$.  Define a metric $d'$ by $d'(x_1,x_2) = d(x_1,x_2) + |\theta(x_2)-\theta(x_1)|$;
as in the proof of  \cite{FP}*{Thm.~3.1}, $\theta$ is Lipschitz with respect to $d'$.  Since $(y,z) \in \arel$, we have that $y >_{d'} z$, and so $\theta(y)\ge\theta(z)$ by \cite{FP}*{Lemma~2.5}.  Since $\theta$ was arbitrary,  we have $(y,z) \in \mathcal Gf$ by \cite{AA}*{p.~51} (note that the opposite inequality convention is used in the definition of Lyapunov function in \cite{AA}, that is, $\theta(f(x))\ge \theta(x)$).  For $y = z$, we show that if $(y,y) \not\in \mathcal Gf$, then $(y,y) \not\in \arel$.  The fact that $(y,y) \not\in \mathcal Gf$ means exactly that $y\not\in \gr(f)$, and so there exists a continuous Lyapunov function $\theta$ with $\theta(f(y)) < \theta(y)$ (\cite{AA}*{Theorem~5}).  Then $y \not\in N(\theta)$, and since $\theta$ is Lipschitz with respect to a compatible metric, we have that $y \not>_\arel y$ by \cite{FP}*{Theorem~2.6}.

\end{proof}

The next theorem, which follows from results in \cite{FP}, shows that we can obtain the generalized recurrent set as the strong chain recurrent set for a particular metric, which is much easier to work with than the intersection over all compatible metrics.

\begin{thm} \label{thm:grisscr}
There exists a metric $\dgr$ compatible with the topology such that $\gr(f) = \scrwo_{\dgr}(f)$.
\end{thm}

\begin{proof}
By \cite{FP}*{Thm.~3.1}, there exists a continuous Lyapunov function $\theta$ for $f$ such that $N(\theta) = \gr(f)$.  Define $\dgr$ by $\dgr(x,y) = d(x,y) + |\theta(y)-\theta(x)|$;
as in the proof of  \cite{FP}*{Thm.~3.1}, $\theta$ is Lipschitz with respect to this metric.  Then, by \cite{FP}*{Thm.~2.6}, $\scrwo_{\dgr}(f) \subset N(\theta) = \gr(f)$.  Since $\gr(f) \subset \scrwo_{\dgr}(f)$ by Proposition~\ref{prop:gre}(\ref{item:fp}), we have $\gr(f) = \scrwo_{\dgr}(f)$.
\end{proof}

\begin{prop}\label{prop:grnotrestrict}
In general, $\gr(f|_{\gr(f)})\subsetneq\gr(f)$.
\end{prop}

\begin{proof}
See Example~\ref{ex:spiralhalf}, or the examples in Theorem~4.2 of \cite{Y}.
\end{proof}

By analogy with Birkhoff's center depth (\cite{Birk}), which involves the nonwandering set, or Yokoi's $*$-depth (\cite{Y}), which involves the strong chain recurrent set, we can define the generalized recurrence depth of $f$ as follows.

\begin{defn}
Let $\gr^0(f) =X$ and $\gr^1(f)=\gr(f)$.  For any ordinal number $\alpha+1$, define $\gr^{\alpha+1}(f)=\gr(f|_{\gr^\alpha(f)})$, and for a limit ordinal $\beta$, define $\gr^\beta(f)=\bigcap_{\alpha<\beta}\gr^\alpha(f)$.  This will stabilize at some countable ordinal $\gamma$, and we define the \emph{generalized recurrence depth} of $f$ to be $\gamma$.  
\end{defn}

The following result follows immediately from work in \cite{Y}.

\begin{prop}
For any countable ordinal $\gamma$, there exists a compact metric space $X_\gamma$ and a continuous map $f_\gamma:X_\gamma\to X_\gamma$ such that the generalized recurrence depth of $f_\gamma$ is $\gamma$.
\end{prop}

\begin{proof}
Yokoi defines $*$-depth as the ordinal at which the sequence $\screc^0(f)=X$, $\screc^1(f)=\screc(f|_{\screc^0(f)})=\screc(f)$, $\screc^2(f)=\screc(f|_{\screc^1(f)}),\dots$ stabilizes, and
 constructs a series of examples to prove that any countable ordinal is realizable as the $*$-depth of some map (\cite{Y}*{Thm.~4.2}).  It is clear that in the examples, $\gr^\alpha(f)=\screc^\alpha(f)$ for all $\alpha$, so these examples also give our result.
\end{proof}

We discuss maps for which the generalized recurrence depth is greater than one (that is, $\gr(f|_{\gr(f)})\subsetneq\gr(f)$) in Section~\ref{sect:relation}.

\section{Generalized recurrence for powers of $f$} \label{sect:powers}

It is well known that $\crec(f^k)=\crec(f)$ for any $k>0$ (see, for example, \cite{AHK}*{Prop.~1.1}). The corresponding statement is not true in general for $\screc(f)$, $\mane(f)$, or the nonwandering set.  (See \cite{Y}*{Ex.~3.4}, or consider Example~\ref{ex:twocircs}:  $\mane(f^2) = \screc(f^2) = C_1\times \Z_2$, while $\mane(f)$ and $\screc(f)$ both equal the entire space $X$.
 See \cites{CN,Saw} for examples for the nonwandering set.)  We now show that it is true for the generalized recurrent set:

\begin{thm}\label{thm:grpower}
For any $k\ge1$, $\gr(f^k)=\gr(f)$.
\end{thm}

\begin{proof}

It is clear that $\gr(f^k)\subset\gr(f)$, so we will prove the opposite inclusion.
We use Theorem~\ref{thm:newgrdef}.  Given any $x\in\gr(f)$ and any sequence $\nseq = \{\prodset{N}_i\}_{i=1}^\infty$  of neighborhoods of  $\diag$, we will find a $(\nseq,f^k)$-chain from $x$ to $x$.  Without loss of generality, we can assume that $\prodset{N}_1\supset\prodset{N}_2\supset\dots$ (if not, replace each $\prodset{N}_i$ by $\bigcap_{j\le i}\prodset{N}_j$).  Define the sets $\prodset{N}_i'$, $i\ge1$, by choosing each $\prodset{N}_i'$ small enough  that $\prodset{N}_i' \subset \prodset{N}_{i-1}'$ ($i>1$) and for any $(\prodset{N}_i',f)$-chain (Defn.~\ref{defn:nchain}) of length $k$ from a point $y$ to a point $z$, we have $(f^k(y),z)\in\prodset{N}_{i}$.  Define  new sequences $\nseq'$ and $\nseq_j'$, $0\le j <k$, by $\nseq'=\{\prodset{N}_{i}'\}_{i=1}^\infty$ and $\nseq_j' =\{\prodset{N}_{ki-j}'\}_{i=1}^\infty$, and note that any $(\nseq_0',f)$-chain is also a $(\nseq_j',f)$-chain for $0<j<k$, as well as a $(\nseq',f)$-chain.
Since $x\in\gr(f)$, there is a $(\nseq_0',f)$-chain $(x_0=x, x_1,\ldots,x_n=x)$
   from $x$ to itself, with $(f(x_{i-1}),x_i) \in \prodset{N}_{k\sigma(i)}'$ ($i=1,\ldots,n$)  for some injection $\sigma:\{1,\ldots,n\} \to\mathbb{N}$.  We may assume that the length $n$ of this chain is a multiple of $k$.  (If not, concatenate it with itself $k$ times, considering the $(j+1)$st copy $(0\le j< k)$ as a $(\nseq_j',f)$-chain; this will be a $(\nseq',f)$-chain.)
   
   For $i=0,k, 2k,\ldots$, define $m_i=\min\{k\sigma(i+1),k\sigma(i+2),\ldots,k\sigma(i+k)\}$.  Then $(x_i,x_{i+1},\ldots,x_{i+k})$ is an $(\prodset{N}_{m_i}',f)$-chain, so  $(f^k(x_i),x_{i+k})\in\prodset{N}_{m_i}$, and  $(x_0=x, x_k,\allowbreak x_{2k},\ldots,x_n=x)$ is  a $(\nseq,f^k)$-chain from $x$ to $x$.

\end{proof}

\section{Relation to ordinary chain recurrence and chain transitivity} \label{sect:relation}

In many cases the generalized recurrent set equals the chain recurrent set.  In this section we give conditions for the two sets to be equal, and discuss what it means for the dynamics if they are not equal.

Yokoi (\cite{Y}) defines a Lyapunov function $\theta$ to be \emph{pseudo-complete} if 
\begin{enumerate}
\item $\theta(f(x))=\theta(x)$ if and only if $x\in\screc(f)$, and
\item $\theta$ is constant on each $d$-strong chain transitive component.
\end{enumerate}

\begin{thm}[{\cite{Y}*{Thm.~5.3}}]
$\screc(f) = \crec(f)$ if and only if there exists a pseudo-complete Lyapunov function $\theta$ for $f$ such that the image $\theta(\screc(f))$ is totally disconnected.
\end{thm}

We obtain a similar statement for $\gr(f)$ using results from \cite{FP}.

\begin{prop}\label{prop:griscr}
$\gr(f) = \crec(f)$ if and only if there exists a Lyapunov function $\theta$ for $f$ such that
	\begin{enumerate}
	\item $\theta(f(x))=\theta(x)$ if and only if $x\in\gr(f)$,
	\item the image $\theta(\gr(f))$ is totally disconnected.
	\end{enumerate}
\end{prop}

\begin{proof}
The ``only if'' direction follows from the existence of a Lyapunov function $\theta$ for $f$ that is strictly decreasing off of $\crec(f)$ and such that $\theta(\crec(f))$ is nowhere dense (\cite{Franks}).
We prove the ``if'' direction.  By hypothesis, $N(\theta) = \gr(f)$.  So $\theta(N(\theta))$ is totally disconnected, and Corollary~1.9 of \cite{FP} implies that $\crec(f)\subset N(\theta)= \gr(f)$.  Since it is always true that $\gr(f)\subset\crec(f)$, we have $\gr(f) = \crec(f)$.

\end{proof}

The following result shows that if the upper box dimension of $\crec(f)$ is small enough, then $\screc(f)=\crec(f)$.  (See \cite{Pesin}*{\S 6} for the definition of upper box dimension, which depends on the choice of metric.)

\begin{thm}
If the upper box dimension of the space $(\crec(f),d)$ is less than one, then two points $x$ and $y$ are chain equivalent if and only if they are $d$-strong chain equivalent.  In particular, $\screc(f)=\crec(f)$.
\end{thm}

Note that the theorem applies in the case that  the space $X$ itself has upper box dimension less than one.

\begin{proof}
If $x$ and $y$ are $d$-strong chain equivalent, they are \emph{a fortiori} chain equivalent, so we will prove the opposite implication.
Let $X_x \subset \crec(f)$ be the chain transitive component containing $x$ and $y$.  Let $D$ be the upper box dimension of $(X_x,d)$.  Define $t_\ep(x,y)$ to be the smallest $n$ such that there is an $\ep$-chain of length $n$ from $x$ to $y$.
It follows from Proposition~22 of \cite{RWcrr} (more precisely, from the discussion in the proof of that result) that there exists a constant $C>0$ (independent of $x$ and $y$) such that for small enough $\ep$, $t_\ep(x,y) \le C/\ep^D$.  Thus, if 
$(x=x_0, x_1, \dots, x_n=y)$ is the shortest $\ep$-chain from $x$ to $y$, then $\sum_{i=1}^n d(f(x_{i-1}),x_i)\le (C/\ep^D)\ep = C\ep^{1-D}$, which goes to zero as $\ep\to0$, and so there is a strong $\ep'$-chain from $x$ to $y$ for any $\ep'$.

\end{proof}

\begin{cor}
Let $\dgr$ be the metric from Theorem~\ref{thm:grisscr} (so  $\scrwo_{\dgr}(f) = \gr(f)$).  If the upper box dimension of the space $(\crec(f),\dgr)$ is less than one, then $\gr(f)=\crec(f)$.
\end{cor}

We will use the following equivalence relation on $\gr(f)$ to help classify maps for which $\gr(f)\ne\crec(f)$.

\begin{defn}
Since the three relations $>_{\mathcal Gf}$, $>_\arel$, and $>_{\prodset C}$  from Theorem~\ref{thm:newgrdef} are identical, they all induce the same equivalence relation on $\gr(f)$, which we will denote by $\sim_f$.

\end{defn}

The quotient space $\gr(f)/\sim_f$ first appears, to the best of my knowledge, in \cite{A}*{Exercise~3.17}.  In \cite{FP}, the equivalence relation $\sim_\arel$ is referred to as ``Mather equivalence.''

\begin{remark}
While $\sim_f$ is an equivalence relation on $\gr(f)$, the chains in the definition(s) are not required to remain in $\gr(f)$.  As we saw in Prop.~\ref{prop:grnotrestrict}, $\gr(f|_{\gr(f)})$
is not necessarily  equal to $\gr(f)$.  And even if the two sets are equal, the  equivalence relations $\sim_f$ and $\sim_{f|_{\gr(f)}}$ may be different, as the following example shows.
.
\end{remark}

\begin{ex}
Let $X_5$ be the disk with the usual topology, and $f_5$ a map that fixes the center point $(0,0)$ and the boundary circle $S_5$ and moves other points in a spiral toward the boundary (see Figure~\ref{fig:spiralcirc}).  Then $\gr(f_5) = \{(0,0)\} \cup S_5$ and $\gr(f_5|_{\gr(f_5)}) = \gr(f_5)$.  There are two $\sim_{f_5}$ equivalence classes, $\{(0,0)\}$ and $S_5$, but each point is its own $\sim_{f_5|_{\gr(f_5)}}$ equivalence class.
\end{ex}

\begin{figure}
\psscalebox{1.0 1.0} 
{
\begin{pspicture}(0,-1.66)(3.28,1.66)
\psarc[linecolor=black, linewidth=0.06, dimen=outer](1.66,0.0){1.6}{0.0}{360.0}
\rput(1.66,0){\psplot[plotstyle=curve, plotpoints=100, xunit=1.0, yunit=1.0, polarplot=true,arrowsize=3pt 2,arrows=->]{1.0}{-724.0}{-1.6 x 20 div sin mul}}
\rput(1.66,0){\psplot[plotstyle=curve, plotpoints=100, xunit=1.0, yunit=1.0, polarplot=true]{-720}{-1890.0}{-1.6 x 20 div sin mul}}
\psdots[linecolor=black, dotsize=0.10](1.66,0)
\end{pspicture}
}
\caption{$f_5:X_5\to X_5$}
\label{fig:spiralcirc}
\end{figure}

However, we do have the following result from \cite{AA}.

\begin{thm}\label{thm:gfct}
The map $f$ restricted to a $\sim_f$ equivalence class is chain transitive.
\end{thm}

\begin{proof}
This follows from applying the second part of \cite{AA}*{Lemma 12} to the $\sim_f$ equivalence class.
\end{proof}

Under what circumstances is $\sim_f$ equivalence different from chain equivalence?  We have a partial answer:

\begin{prop}\label{prop:identity}
Let $f$ be chain transitive on an invariant subset $N$ of $\gr(f)$, and assume that $x \not\sim_f y$ for some pair of points $x$ and $y$ in $N$.  
Then $N/\sim_f$ is a nontrivial connected set, and the factor map $N/\sim_f \to N/\sim_f$ is the identity.
\end{prop}

\begin{proof}
Let $M$ be the quotient space $N/\sim_f$, and $\pi:N\to M$ the projection.  By hypothesis, $M$ contains more than one point.  Since the $\sim_f$ equivalence classes are $f$-invariant, $f|_N$ induces the identity map on $M$.  Assume that $M$ is not connected, and let $U$, $V$ be a separation of $M$.  Then $\pi^{-1}(U)$, $\pi^{-1}(V)$ is a separation of $N$.  Since $f(\pi^{-1}(U))\subset \pi^{-1}(U)$ and $f(\pi^{-1}(V))\subset \pi^{-1}(V)$, there is no $\ep$-chain from any point in $\pi^{-1}(U)$ to any point in $\pi^{-1}(V)$ for any $\ep < d(\pi^{-1}(U), \pi^{-1}(V))$, contradicting chain transitivity.

\end{proof}

In the examples that we have seen where the chain recurrent set is strictly larger than  the generalized recurrent set, the difference was in some sense caused by the presence of a large set of fixed points (either an interval or a Cantor set).  However, the two sets can be different even if there are no fixed points, as the following example shows.

\begin{ex}
Consider the map $f=f_1\times \rho$ on the torus $S^1\times S^1$, where $f_1$ is the map from Example~\ref{ex:halfcirc} and $\rho$ is an irrational rotation.  Then $\crec(f) =S^1\times S^1$, while $\gr(f) = C_1 \times S^1$.
\end{ex}

However, the map in this example factors, by projection onto the first coordinate, onto a map with many fixed points.  This observation leads to the following characterization of maps for which the generalized recurrent set is strictly contained in the chain recurrent set.

\begin{thm}\label{thm:unctble}
If $\gr(f) \ne \crec(f)$, then $f$ factors onto a map with uncountably many fixed points.
\end{thm}

\begin{proof}
Theorem~3.1 of \cite{FP} tells us that there is a Lyapunov function $\theta:X\to\R$ for $f$ such that $\theta(f(x))=\theta(x)$ if and only if $x\in\gr(f)$, so, by Proposition~\ref{prop:griscr}, we must have that the image $\theta(\gr(f))$ contains an interval.  Proposition~3.2 of \cite{FP} says that $\theta$ is constant on each $\sim_f$ equivalence class, so $\theta$ induces a map $\bar\theta$ on the  quotient $\gr(f)/\sim_f$.  Since the image $\bar\theta (\gr(f)/\sim_f) =\theta(\gr(f))$ contains an interval, we must have that $\gr(f)/\sim_f$ is uncountable.  If, as in \cite{AA}, we extend the  equivalence relation $\sim_f$ from $\gr(f)$ to an equivalence relation $\sim$ on all of $X$ by setting $\sim\ =\ \sim_f \cup \diag$ (that is, $x\sim y$ if $x=y$ or $x\in\gr(f)$, $y\in\gr(f)$, and $x\sim_f y$), then $f$ factors onto the map $\bar f: X/\sim \to X/\sim$, with fixed points $\gr(f)/\sim_f=\gr(f)/\sim$.

\end{proof}

\begin{cor}
If $\gr(f) \ne \crec(f)$, then either
	\begin{enumerate}
	\item $\gr(f)/\sim_f$ contains a nontrivial connected set, or
	\item $\gr(f)/\sim_f$ is homeomorphic to the disjoint union of a Cantor set and a countable set.
	\end{enumerate}

\end{cor}

\begin{proof}
The Cantor-Bendixson theorem (\cite{Sierp}*{Thm.~47}) says that $\gr(f)/\sim_f$ can be written as the disjoint union of a perfect set $P$ and a countable set.  Since $\gr(f)/\sim_f$ is uncountable, the set $P$ must be nonempty.  If $\gr(f)/\sim_f$ does not contain a nontrivial connected set, then it is totally disconnected, and so $P$ is a nonempty, totally disconnected, compact, perfect set, that is, a Cantor set.
\end{proof}

\begin{cor}
If the generalized recurrence depth of $f$ is greater than one (that is, if $\gr(f|_{\gr(f)})\subsetneq\gr(f)$), then $f$ factors onto a map with uncountably many fixed points.
\end{cor}

\begin{proof}
It follows from Theorem~\ref{thm:gfct} that
 $\crec(f|_{\gr(f)})= \gr(f)$.  So we can apply the reasoning in Theorem~\ref{thm:unctble} to the map $f|_{\gr(f)}:\gr(f)\to\gr(f)$.  We extend the  equivalence relation $\sim_{f|_{\gr(f)}}$ from $\gr(f|_{\gr(f)})$ to all of $X$ by setting $\sim = \sim_{f|_{\gr(f)}} \cup \diag$; the induced map on $X/\sim$ will have the uncountable set $\gr(f|_{\gr(f)})/\sim$ as the fixed point set.

\end{proof}

\end{document}